\documentclass[11pt]{article}
\usepackage{amsfonts}
\usepackage{amsfonts}  
\usepackage{amsthm}
\usepackage{amsmath}
\usepackage{color}

\title{${\cal C}^{1, \gamma} $ regularity for singular or degenerate fully nonlinear operators and applications} 

\author{Isabeau Birindelli \\
Dipartimento di Matematica, Sapienza Universit\`a\  di Roma
\and
  Fran\c{c}oise Demengel\\
  D\'epartement de Math\'ematiques,
Universit\'e\  de Cergy-Pontoise
  \and
   Fabiana  Leoni\\ 
   Dipartimento di Matematica, Sapienza Universit\`a\  di Roma}

\date{}

\catcode`@=11 \@addtoreset{equation}{section} \catcode`@=12

\newtheorem{theo}{Theorem}[section]
\newtheorem{prop}[theo]{Proposition}
\newtheorem{rema}[theo]{Remark}

\def\R{\mathbb  R}
\def\grad{\nabla}

\begin{document}
\maketitle
\begin{abstract}
In this note,   we prove $\mathcal{C}^{1,\gamma}$ regularity  for solutions
of some fully nonlinear degenerate elliptic equations with  "superlinear"  and "subquadratic " Hamiltonian terms.  As an application, we  complete  the results of \cite{BDL1}  concerning the associated ergodic problem,  proving, among other facts,  the uniqueness, up to constants, of the ergodic function. 
\end{abstract} 

\section{Introduction}

The goal of the present paper is to establish $C^{1,\gamma}$ regularity results and to obtain a priori estimates for viscosity solutions of a class of fully nonlinear elliptic equations which may be singular or degenerate at the points where the gradient of the solution vanishes. 
\medskip

Regularity properties of viscosity solutions of fully nonlinear elliptic equations have been studied since a long time, starting with the seminal paper of Caffarelli \cite{Caf} in 1989, which contains in particular ${ \cal C}^{1, \gamma}$ estimates for $F( x, D^2 u) = f$ when $f \in L^p$, $p> n$. His results were extended to $L^p$ viscosity solutions of operators $F(x, D^2 u, Du)$  at most linear in the gradient   by Swiech in \cite{S}. Later, Winter \cite{W} proved ${ \cal C}^{1, \gamma} ( \overline{ \Omega}) $ estimates in the presence of  a   regular boundary  datum. 

 In the recent preprint  \cite{SN}, Saller Nornberg proves ${ \cal C}^{1, \gamma}$ and $W^{2,p}$ results for $L^p$ viscosity solutions, when $F(x, D^2 u, Du)$ is fully nonlinear, uniformly elliptic and  at most quadratic  in the gradient. 
 
  In \cite{BD1}, the  first two authors of the present paper consider singular or degenerate equations of the form
  $$| \nabla u |^\alpha F( D^2 u) = f(x, \nabla u)\, ,$$
   where $F$ is fully nonlinear  uniformly  elliptic,  $\alpha > -1$,  and $f$  has growth  at most of order $1+ \alpha$ in the gradient. 
    Lipschitz regularity results are proved in \cite{BD1}, and ${ \cal C}^{1, \gamma}$ regularity   up to the boundary in  the case $\alpha \leq 0$ in \cite{BD2}, for the   Dirichlet problem with homogeneous boundary conditions.  
     Later, ${ \cal C}^{1, \gamma}$ interior regularity  was obtained in the case $\alpha >0$  by Imbert and Silvestre \cite{IS}, when $f$ does not depend on  $\nabla u$. These results have been  extended to  the case where $f$ depends on the gradient with  growth at most $\alpha+1$,  and to boundary ${ \cal C}^{1, \gamma}$ results in the presence of sufficiently regular  boundary datum, in \cite{BDcocv}, \cite{BD3}, \cite{BD4}.  
     \medskip
     
      In this note we prove ${ \cal C}^{1, \gamma}$ interior and boundary  regularity results  when the equation possesses  some Hamiltonian "superlinear" but at most  "quadratic" in the gradient.  More precisely, 
 we consider equations of the form
\begin{equation}\label{eq1}
- |\nabla u|^\alpha F(  D^2 u)+ b(x) | \nabla u |^\beta   = f(x)\, ,
\end{equation}
 where the coefficient functions $b$ and $f$ are continuous, and the exponents $\alpha$ and $\beta$ always satisfy  $\alpha>-1$ and $\alpha+1<\beta \leq \alpha +2$. On the second order operator  $F$, we assume it is a continuous function 
 $F:{\mathcal S}_N\to \R$  defined on the set ${\mathcal S}_N$  of $N\times N$ symmetric matrices, positively homogeneous of the degree one, and satisfying further the uniform ellipticity condition 
\begin{equation}\label{eqF1}
a \, {\rm tr}(N) \leq F(M+N)-F( M) \leq  A \, {\rm tr}(N)\, ,
\end{equation} 
 for any $M, N\in {\mathcal S}_N$, with $N\geq 0$,  for given positive constants $A\geq a>0$.
 
 \noindent The considered equations include, as a very special case, the semilinear equation
 $$
 -\Delta u + b(x) |Du|^\beta= f(x)
 $$
 with $1<\beta\leq 2$, and it is for this reason that   the growth of  the first order terms of equations \eqref{eq1} is referred to as "superlinear" and "subquadratic".
 
\noindent The definition of viscosity solution we adopt is the one   firstly  introduced  in \cite{BD1},  which is equivalent to the usual one in the case $\alpha \geq 0$, and, in  any cases, allows not to test points where the gradient of the test function is zero, except in the locally constant case.

In the  paper \cite{BDL2}, we  proved local and global Lipschitz regularity results for viscosity solutions of \eqref{eq1}. In particular, we showed that if $u$ satisfies in the viscosity sense equation \eqref{eq1} in a domain $\Omega\subseteq \R^N$, then, for any pair of bounded subdomains $\omega\subset\subset \omega'\subset\subset \Omega$, there exists a positive constant $M$ depending on $a, A, N, \alpha, \beta, \|u\|_{L^\infty (\omega')}, \|f\|_{L^\infty (\omega')}, \|b\|_{W^{1, ^\infty} (\omega')}$ and on $ \omega, \omega'$ such that
$$
|\nabla u(x)|\leq M\qquad \hbox{a.e. in } \omega\, .
$$
Our main result in the present paper reads as follows.
\begin{theo}\label{C1reg}Suppose that $\Omega$ is an open subset in $\R^N$, that $f \in { \cal C} ( \Omega)$, and $b \in W^{1, \infty}_{loc} ( \Omega)$. 
Let $u\in C(\Omega)$ be a viscosity solution of \eqref{eq1} in $\Omega\subseteq \R^N$. Then, there exists $0<\gamma\leq {1\over 1+ \alpha^+}$ depending on the data, such that $u$ belongs to $C^{1,\gamma}_{loc}(\Omega)$ and, moreover, for any pair of subsets $\omega\subset\subset \omega'\subset\subset \Omega$ one has
\begin{equation}\label{bound1}
|\nabla u|_{C^{0,\gamma}(\omega)}\leq C_1 \left( |u|_{L^\infty(\omega')} , |b|_{ W^{1, \infty}( \omega^\prime)}, |f| _{L^\infty(\omega')}\right)\, .
\end{equation}
\end{theo}
Note that some more  explicit bound is the following, depending on the Lipschitz norm of $u$: 
$$
|\nabla u|_{C^{0,\gamma}(\omega)}\leq C \,  \left( |u|_{W^{1, \infty }(\omega')} +  |b|_\infty^{1\over 1+ \alpha}  |u|_{W^{1, \infty }(\omega')}^{\beta\over 1+ \alpha} +  |f|^{1\over 1+ \alpha}  _{L^\infty(\omega')}\right)\, .
$$

Theorem \ref{C1reg} covers the case  $\beta \leq \alpha+1$, treated in \cite{BDcocv}. We remark that  the arguments used in \cite{BDcocv}  are  different and fail  in the case $\beta > \alpha+1$. 
     
 Furthermore, Theorem \ref{C1reg} include  the results of \cite{SN} for $\alpha = 0$. We observe that in \cite{SN} the author uses essentially the ABP estimate of \cite{KS} for fully nonlinear elliptic equations with  quadratic growth in the gradient, and Caffarelli's iterative method.        This method cannot be employed when $\alpha \neq 0$.  The main ingredients in the proof of Theorem \ref{C1reg} are the Lipschitz continuity of  solutions, a  fixed point  argument, the  existence and uniqueness result for Dirichlet problem proved in \cite{BDL2}, and  the   ${ \cal C}^{1, \gamma}$  estimates   \cite{Caf}  for $\alpha = 0$,     of      \cite{IS}   and \cite{BD3}   for   $\alpha >0$, and the one proved in   Proposition \ref{c1gamma}   when $\alpha <0$ . 
       
         As an application of the $C^{1,\gamma}_{loc}$ regularity of viscosity solutions of equation \eqref{eq1}, we prove the uniqueness of the ergodic function associated with the considered operators. Let us recall that, as recently proved in \cite{BDL1},  if $\Omega\subset \R^N$ is an open, bounded, $C^2$ domain, if $d(x)$ denotes the distance function from $\partial \Omega$ and if the operator $F$ satisfies the extra regularity assumption
  \begin{equation}\label{smooth}
F(\nabla d(x)\otimes \nabla d(x)) \hbox{ is $\mathcal{C}^2$ in a neighborhood of } \partial \Omega\, ,
\end{equation}
then, given a locally Lipschitz continuous datum $f$, there exists a unique constant $c_\Omega$, called the ergodic constant or additive eigenvalue of $F$, such that the infinite boundary condition problem
\begin{equation}\label{ergodicp}
\left\{ \begin{array}{lc}
-| \nabla u |^\alpha F(D^2 u) + |\nabla u|^\beta   = f +  c_\Omega  & \hbox{ in} \ \Omega \\
u=+\infty &  \hbox{ on} \ \partial \Omega 
\end{array}\right.
\end{equation}
has a solution $u\in C(\Omega)$, called ergodic function.

In Section 3 we will prove the following result. 
\begin{theo}\label{unique} Let $\Omega\subset \R^N$ be a bounded domain of class $C^2$ and let $F$  satisfy  \eqref{eqF1} and \eqref{smooth}. Assume further that $\alpha >-1$, $\alpha+1<\beta < \alpha +2$ and that $f\in C(\Omega)$ is  bounded and locally Lipschitz continuous. Then,  problem (\ref{ergodicp}) has a  unique (up to additive constants) solution, provided that,  when   $\alpha \neq 0$, 
 $\sup_\Omega f <- c_\Omega$ and $\partial \Omega$ is connected.  
\end{theo}

\section{Proof of Theorem \ref{C1reg}.}
In all this section we set $|v|_{ W^{1, \infty}} : = | v |_\infty+ | \nabla v |_\infty$. 

We  begin by proving some ${ \cal C}^{1, \gamma}$ interior estimate  for the case $\beta = 0$, which completes the result in \cite{BD2}. The proof is very similar to the one employed  in \cite{BD2}, but we reproduce it here for the sake of completeness. 
  We denote by $C_{lip}
 $ some constant such that, by  \cite{BD1},
              for any $u\in W^{1, \infty} ( B(0,1))$  and for any $g$ continuous and bounded in $B(0,1)$, any  solution  $w$ of 
      \begin{equation} \label{lip}
      \left\{\begin{array}{cl}
       -| \nabla w |^\alpha F( D^2 w) = g & {\rm in } \ B(0,1)\\
        w= u & {\rm on } \ \partial B( 0,1)
        \end{array}\right.
        \end{equation} satisfies 
         \begin{equation} \label{lipalpha}
      |w|_{ W^{1, \infty}(B(0,1))}\leq  C_{lip}  ( |u|_{ W^{1, \infty}(B(0,1))} + |g|_{L^\infty( B(0,1))}^{1 \over 1+ \alpha} ). 
         \end{equation}
         
           We then have the following 
   \begin{prop}\label{c1gamma}
  Under the above assumptions,      
  any   solution $w$ of (\ref{lip})  satisfies : 
 for any $r<1$, there exists $C_r $ such that 
            \begin{equation} \label{Cr}|w|_ {{ \cal C}^{1, \gamma}( B(0, r))}  \leq  C_r ( |u|_{W^{1, \infty}(B(0,1))}+ |g|_{L^\infty( B(0,1))}^{1\over 1+ \alpha})
         \end{equation}
          \end{prop}
          
           \begin{proof} 
           This result is well known in the case $\alpha = 0$, \cite{Caf}, while the case $\alpha >0$ is proved in \cite{IS}. It remains to consider the case $\alpha <0$.
               
      Take $\epsilon = {1 \over 4C_{lip}}$ and $\delta \leq 1 $. 
                 Let 
         $${ \cal K}_R = \{ v \in { \cal C}^1( B(0,1)) \cap W^{1, \infty}(B(0,1)) , | v|_{ W^{1, \infty}} \leq R\}$$
          where  $R$ is chosen large enough, fixed such  that $R\geq |w|_{ W^{1, \infty}}$ and 
          $$R  \geq 2\, C_{lip} ( |f|_\infty  ( 1 
+ R^{-\alpha}) + { R^{1+\alpha }\over 2C_{lip}})+ |u|_{ W^{1, \infty}})$$
           which is possible since both $-\alpha$ and $1+ \alpha$ are lesser than 1. 
           
           We define the map ${ \cal T}: 
            v\mapsto w_\delta$ where $w_\delta$ satisfies 
            $$\left\{ \begin{array}{cl}
   - F(D^2 w_\delta ) =( f -\epsilon |v|^\alpha v + \epsilon |w|^\alpha w) ( \delta^2 + | \nabla v|^2)^{-\alpha \over 2}  & {\rm in} \ \ B(0,1)\\
    w_\delta  = u & {\rm on} \ \partial B(0,1) .
    \end{array}\right.$$
    The map $\mathcal{T}$ is well defined since the right hand side is continuous and bounded. 
     Furthemore, by  (\ref{lipalpha}) in the case $\alpha = 0$, one has 
      \begin{eqnarray*}
       |w_\delta |_{ W^{1, \infty}} &\leq &C_{lip} \left( \left\vert ( f -\epsilon |v|^\alpha v + \epsilon |w|^\alpha w ) ( \delta^2 + | \nabla v|^2)^{-\alpha \over 2}   \right\vert_\infty + |u|_{W^{1, \infty}}\right) \\
       & \leq&C_{lip}\left(  ( |f|_\infty +2 {
      R^{1+ \alpha }  \over 4C_{lip}} )( \delta^{-\alpha} + R^{-\alpha} )\ +  |u|_{W^{1, \infty}}\right)\\
      &\leq & C_{lip}\left( |f|_\infty ( 1+ R^{-\alpha} )+ {R \over 2C_{lip}} +  {
      R^{1+ \alpha }  \over 2 C_{lip}} + |u|_{ W^{1, \infty} }\right)\\
       &\leq & R
       \end{eqnarray*}
        by the  choice of $R$. hence, ${\cal K}_R$ is a closed convex set invariant for  ${ \cal T}$. Moreover, $\mathcal{T}$ is a compact operator ( see \cite{BD1}) so that, by  Schauder's theorem,  ${ \cal T}$ possesses a fixed point denoted by $\overline{w_\delta} $. 
        Note that $\overline{w_\delta} $  has a Lipschitz norm independent of $\delta$ : indeed, using the convexity inequalities 
         $$\epsilon \delta^{-\alpha} |\overline{w_\delta} |_\infty^{1+ \alpha} \leq ( 1+ \alpha) \epsilon  |\overline{w_\delta} |_\infty+ (-\alpha) ( \delta \epsilon)^{-\alpha}$$
          and 
          $$| f + \epsilon |w|^\alpha w |_\infty  | \overline{w_\delta} |^{-\alpha} \leq (-\alpha) \epsilon |\overline{w_\delta} |_\infty +  \epsilon^{\alpha\over 1+ \alpha} (| f |_\infty  \epsilon |w|_\infty ^{1+\alpha})^{1 \over 1+ \alpha}\, ,$$
           one gets that 
           $$|\overline{w_\delta} |_{W^{1, \infty}} \leq2 C_{lip} \left(   \epsilon^{\alpha\over 1+ \alpha} (| f |_\infty  \epsilon |w|_\infty ^{1+\alpha})^{1 \over 1+ \alpha}+ (-\alpha) ( \delta \epsilon)^{-\alpha} +  |u|_{W^{1, \infty}}\right).$$
           
            Furthermore by (\ref{Cr}) in the case $\alpha = 0$,  $\overline{w_\delta} $ satisfies 
            \begin{eqnarray}\label{c3r}
      |\overline{w_\delta} |_{ {\cal C}^{1, \gamma}( B(0,r)} &\leq& C_r  \left(  ( |f|_\infty +{ |w |_{ W^{1, \infty}} \over 4C_{lip} }+  { |\overline{w_\delta} |_{ W^{1, \infty}} \over 4C_{lip} }) (1 + |\overline{w_\delta} |_{ W^{1, \infty}}  ^{-\alpha})+  |u|_{W^{1, \infty}}\right)\nonumber\\
           \end{eqnarray}
      Note  that $\overline{w_\delta}$ satisfies 
      $$\left\{\begin{array}{lc}
      -( \delta^2 + | \nabla \overline{w_\delta}|^2)^{\alpha \over 2} F( D^2 \overline{w_\delta}) + \epsilon |\overline{w_\delta}|^\alpha \overline{w_\delta} = f + \epsilon |w|^\alpha w& {\rm in} \ B(0, 1)\\
      w_\delta = u & {\rm on} \ \partial B(0,1) 
    \end{array}\right.$$
       
       Using the estimate (\ref{c3r}) which is independant on  $\delta$,  up to subsequence,  $w_\delta$ converges locally uniformly  when $\delta$ goes to zero,  towards a solution $\overline{ w}$ of 
      $$\left\{ \begin{array}{lc}
      -| \nabla \overline{w} |^\alpha F( D^2 \overline{w}) + \epsilon |\overline{w}|^\alpha \overline{w}= f+ \epsilon |w|^\alpha w& {\rm in} \ B(0,1)\\
       w= u & {\rm on } \ \partial B(0,1) \end{array} \right.$$
   
     By uniqueness of solutions of such equation ( see \cite{BD1}), one gets that 
       $\overline{ w} = w $ and then   by (\ref{c3r}),   $w$ is ${ \cal C}^{1, \gamma}$ in $B(0, r)$. To get the precise estimate (\ref{Cr}) let us observe that 
        since $w$ is ${ \cal C}^1$, then it is a solution of 
        $$\left\{ \begin{array} {lc}
         -F( D^2 \varphi) = | \nabla w |^{-\alpha} f &\ {\rm in} \ B(0,1)\\
          \varphi = u & {\rm on} \ \partial B(0,1)
          \end{array}\right.$$
           In particular one has  by ( \ref{Cr}) in the case $\alpha = 0$: 
           \begin{eqnarray*} |w|_{ { \cal C}^{1, \gamma} ( B(0, r))} &\leq& C_r (  | \nabla w |_\infty ^{-\alpha} |f|_\infty + |u|_{ W^{1, \infty}} )\\
           &\leq & C_r ( |f|_\infty^{1 \over 1+ \alpha} + |w|_{ W^{1, \infty}} + |u|_{ W^{1, \infty}})\\
           &\leq & C_r (1+ C_{lip}) (  |f|_\infty^{1 \over 1+ \alpha} + |u|_{ W^{1, \infty}})
           \end{eqnarray*}
            So we get (\ref{Cr}) with $C_r $ replaced by $ C_r ( 1+ C_{lip})$. In the following we will denote  for simplicity  $C_r $    this constant.         
      
\end{proof}

 We now   recall the Lipschitz estimate proved in \cite{BDL2}.
 \begin{theo}\label{theolip} Suppose that $F$ is uniformly elliptic,  $f$ is continuous  in $B(0,1)$, $b$ is locally Lipschitz continuous in $B(0,1)$,   $\alpha>-1$ and   $\beta\in (0,\alpha+2]$. Let  $u$ be a  locally bounded   viscosity solution of 
      $$- | \nabla u |^\alpha F(D^2 u)  + b(x)|\nabla u|^\beta =   f(x)\quad \hbox{ in } B(0,1) $$
 Then $u$ is  locally Lipschitz continuous in $B(0,1)$, that is,  for any  $r <r^\prime <1$,  there exists some constant $c$ depending on the ellipticity constants of $F$, on $r$, $r'$ and on universal constants, such  that 
 $$|u|_{W^{1, \infty}(B(0,r))} \leq c( |u|_{L^\infty( B(0, r^\prime))}  , |f|_{L^\infty( B(0, r^\prime))} , |b|_{W^{1, \infty}(B(0, r^\prime))} )$$
  Furthermore, if $u$ and $f$ are bounded and $b$ is bounded and Lipschitz continuous in $B(0,1)$, then $u$ is Lipschitz continuous in $B(0,1)$  and there exists $c$ such that 
  $$|u|_{W^{1, \infty}(B(0,1))} \leq c( |u|_{L^\infty( B(0, 1))}  , |f|_{L^\infty( B(0, 1))} , |b|_{W^{1, \infty}(B(0, 1))} )$$
      \end{theo}

       \begin{rema}
  {\rm  The assumption that  $b$
 is  Lipschitz continuous   is needed only in the case $\beta = \alpha+2$. For the case $\beta < \alpha+2$, $b$ bounded is sufficient.}
 \end{rema}

  Since we want to prove local estimates inside $B(0, r)$, we can suppose that $u$ is Lipschitz continuous on  the whole of $B(0,1)$ and we will set   $M= |u|_{ W^{1, \infty}(B(0,1)} $. 
  
\emph{Proof  of Theorem \ref{C1reg}}. 
   We will use  both a truncation method, and a fixed point argument. The previous estimate (\ref{Cr}) will then enable  us to say that  a  solution  of (\ref{eq1})
    is locally  ${ \cal C}^{1, \gamma}$.

               In the following,  $T_M$ will denote the truncation operator at the level $M$, more precisely $T_M ( s) = \inf \{ |s|, M\} {s \over |s|}$.

          Let  $\epsilon_\alpha = {2^{-1-{(-\alpha)^+ \over 1+ \alpha}}\over C_{lip}}$. Let  us observe that using  (\ref{lipalpha}), when $g$ is continuous and bounded,   the unique solution $w$, ( see \cite{BD1}) of 
          $$ \left\{ \begin{array}{lc}
           -| \nabla w | ^\alpha F( D^2 w) + \epsilon_\alpha^{1+ \alpha}  |w|^\alpha w = g& {\rm in} \ B(0,1)\\
           w = u & {\rm on} \ \partial B(0,1)
           \end{array}
           \right.$$
            satisfies 
            $$|w|_{ W^{1, \infty}( B(0,1)} \leq 2^{1+ ({-\alpha)^+ \over 1+ \alpha}}C_{lip} (|g|_\infty^{1 \over 1+ \alpha}+ |u|_{ W^{1, \infty}})$$
             Indeed, by (\ref{lipalpha})
             \begin{eqnarray*}
              |w|_{ W^{1, \infty}( B(0,1)} &\leq &C_{lip} ( |g-\epsilon_\alpha ^{1+ \alpha}  |w|^\alpha w |_\infty^{1 \over 1+ \alpha} + |u|_{ W^{1, \infty}}))\\
              &\leq &2^{ (-\alpha)^+ \over 1+ \alpha} C_{lip}( |g|_\infty^{1 \over 1+ \alpha} + \epsilon_\alpha |w|+ |u|_{ W^{1, \infty}}))\\
              &\leq & {  |w|_{ W^{1, \infty}( B(0,1)} \over 2} + 2^{ (-\alpha)^+ \over 1+ \alpha} C_{lip}( |g|_\infty^{1 \over 1+ \alpha}+ |u|_{ W^{1, \infty}}).
              \end{eqnarray*}
              Let 
           $$R= 2^{1+ {(-\alpha)^+ \over 1+ \alpha}}C_{lip} \left(\left\vert  |f|_\infty + \epsilon_\alpha^{1+ \alpha} |u|_\infty ^{1+ \alpha}\right\vert ^{1 \over 1+ \alpha} +  (|b|_\infty M^{ \beta} )^{ 1 \over 1+ \alpha}  + |u|_{ W^{1, \infty} ( B(0,1)} \right) , $$
           define 
           $${\cal K}_R = \{ v \in { \cal C}^1 \cap W^{1, \infty} ( B(0,1)), |v|_\infty + | \nabla  v|_\infty\leq R\}, $$
           and note that ${ \cal K}_R$ is a closed convex set in $ { \cal C}^1 \cap W^{1, \infty} ( B(0,1))$. 
           We also define the map  ${\cal T} : v \rightarrow w$ where $w$ is  the unique  solution, ( see \cite{BD1}) of 
             
$$
      \left\{\begin{array}{lc}
       - | \nabla w |^{ \alpha} F( D^2 w) + \epsilon_\alpha ^{1+ \alpha} |w|^\alpha w =  \left( f+\epsilon_\alpha ^{1+ \alpha}  |u|^\alpha u  - b(x) T_M ( | \nabla v |)^\beta \right)& {\rm in } \ B(0,1)\\
        w= u & {\rm on } \ \partial B( 0,1)
        \end{array}\right.
      $$
         $w$ is well defined since the right hand   is continuous and bounded. By ( \ref{lipalpha} )  $w$ satisfies 
         \begin{eqnarray*}
          |w|_{ W^{1, \infty} (B(0,1)} &\leq& C_{lip} \left(  (|f|_\infty + \epsilon_\alpha ^{1+ \alpha}  |u|_\infty^{1+ \alpha} + |b|_\infty M^\beta)^{1\over 1+ \alpha}  +  |u|_{ W^{1, \infty} ( B(0,1)}\right)\\
          &\leq & 2^{ {(-\alpha)^+ \over 1+ \alpha}}C_{lip} \left(  (|f|_\infty + \epsilon_\alpha ^{1+ \alpha}  |u|_\infty^{1+ \alpha} )^{1\over 1+\alpha} + ( |b|_\infty M^\beta )^{1\over 1+ \alpha} +   |u|_{ W^{1, \infty} }\right)\\
                   &\leq R
          \end{eqnarray*}
          by the choice of $R$,  hence ${ \cal T}$ sends ${\cal K}_R$ into itself. Furthermore by classical uniform estimates,(see \cite{BD1}),  ${ \cal T}$ is a compact operator.  Then, by Schauder's fixed  point Theorem,  it possesses  a fixed point $\overline{w}$ which then satisfies 
          \begin{equation}\label{eqalphaunic} \left\{\begin{array}{lc}
       -| \nabla \overline{w} |^\alpha F( D^2 \overline{w}) + b(x) |T_M ( \nabla \overline{w})|^\beta + \epsilon_\alpha ^{1+ \alpha} |\overline{w}|^\alpha\overline{w} = f+ \epsilon_\alpha ^{1+ \alpha}|u|^\alpha u  & {\rm in } \ B(0,1)\\
        w= u & {\rm on } \ \partial B( 0,1)
        \end{array}\right.
        \end{equation} 
        Recall that $| \nabla u | \leq M$,  then $u$ satisfies the same equation. By a mere adaptation of the  comparison principle in \cite{BD1}, there is uniqueness  of  the solution to ( \ref{eqalphaunic}) ,  hence  $u= \bar w$ and by (\ref{Cr}) one gets that $u$ is ${ \cal C}^{1, \gamma}$ for the $\gamma$  allowed by (\ref{Cr}). 
         
\hfill$\Box$
          
         \begin{rema}  
{\rm Let us observe that, in the case $\alpha \leq 0$, and  if the operator $F$ is convex or concave in the Hessian argument, then we can repeat the above proof with the $C^{1,\gamma}_{loc}$ estimates replaced by the $W^{2,p}_{loc}$ estimates, for any $1<p<\infty$, see \cite{Caf, CCKS}. This yields the local a priori estimate, for any $\omega \subset \subset \omega^\prime \subset \subset \Omega$
$$
\|u\|_{W^{2,p}(\omega)}\leq C_p \left( \|u\|_{L^\infty(\omega')} , \|f\|_{L^\infty(\omega')}, |b|_{ W^{1, \infty}( \omega^\prime)}\right)
$$
for any viscosity solution $u$ of equation \eqref{eq1}.

Furthermore, suppose that $\Omega\subset \R^N$ is a  smooth  bounded domain, that $f \in { \cal C} ( \overline{ \Omega})$, $b \in W^{1, \infty} ( \Omega)$,   that  $u$ is a viscosity solution  of \eqref{eq1} in $\Omega$,  which  satisfies the boundary condition $u=\psi$ on $\partial \Omega$, with $\psi\in C^{1,\gamma_0}(\partial \Omega)$, by using the up to the boundary estimates of \cite{W}, \cite{BDcocv}, we obtain the global regularity bound
$$
\|u\|_{C^{1,\gamma}(\overline{\Omega})}\leq C  \left( \|\psi\|_{C^{1,\gamma}(\partial \Omega)} , \|f\|_{L^\infty(\Omega)},  |b|_{ W^{1, \infty}( \Omega)}\right)
$$
for some exponent $\gamma\leq \inf ( \gamma_0, {1 \over 1+ \alpha^+}) $.}

\end{rema}

\medskip

 \begin{rema}
  {\rm As  in \cite{SN}, 
   we can extend our results  (only  for  $\alpha \leq 0$) to the case where  $F(M)$ is replaced by 
    $F( p, M)$,  satisfying the following structural assumptions: there exist positive constants $\mu , b$  such that 
     for any $(p, q)\in (\R^N)^2$, $(X, Y)\in ({\cal S}_N)^2$ one has 
      \begin{eqnarray*}\label{SC}
      -b | p-q|-\mu ( |p|+ |q|) ( |p-q|)
       &+& { \cal M}^-( X-Y) \\
      &\leq& F(  p, X)-F(  q, Y)\\
      & \leq& { \cal M}^+( X-Y) + b | p-q|+\mu ( |p|+ |q|) ( |p-q|) 
      \end{eqnarray*}
     where ${ \cal M}^+$ and ${\cal M}^-$ denote the Pucci extremal operators.     The fixed point argument  and the truncation method can be easily  adapted to this case.}
   \end{rema}
   
\section{Gradient estimates and  proof of Theorem \ref{unique}.}

As an application of the regularity results proved in the previous section, we now focus on the  ergodic pairs associated to the class of operators we are considering. 

Precisely, 
given a  function $f\in \mathcal{C}(\Omega)$, let  $c_\Omega\in \R$  be a  constant for which there exist  solutions $u\in \mathcal{C}(\Omega)$ of the infinite boundary value problem \eqref{ergodicp}.

\noindent In \cite{BDL1} we gave sufficient conditions for the existence of  ergodic pairs $(c_\Omega , u)$ which solve (\ref{ergodicp}),  and we proved the uniqueness of $c_\Omega$ in some cases. Here, we are concerned in particular with the uniqueness of  $u$.

As proved in \cite{BDL1} and recalled in the introduction, the rate of boundary explosion  of any ergodic function can be made precise assuming that the operator $F$ satisfies the "boundary"  regularity condition  
\begin{equation}\label{C2reg}
F(\nabla d(x)\otimes \nabla d(x)) \hbox{ is a $\mathcal{C}^2$ function in a neighborhood of $\partial \Omega$\, .}
\end{equation}
Here, $d(x)$ denotes the distance function from $\partial \Omega$, and it is of class $\mathcal{C}^2$ in a neighborhood of $\partial \Omega$ by the regularity assumption on the domain. Condition    \eqref{C2reg} is certainly satisfied if the domain $\Omega$ is of class ${\mathcal C}^3$ and the operator $F$ is  ${\cal C}^2$, but there can be also cases with non smooth $F$ satisfying \eqref{C2reg} in ${\mathcal C}^2$ domains. For instance, when  $F(M)$  depends only on the eigenvalues of $M$, as in the case of  Pucci's operators,  $F(\nabla d(x)\otimes \nabla d(x))$ is a constant function as long as $|\nabla d(x)|=1$. 

Under assumptions  \eqref{eqF1} and \eqref{C2reg} on $F$, and if $f\in {\mathcal C}(\Omega)$ is locally Lipschitz continuous, bounded from below and satisfying
$$
\lim_{d(x)\to 0} f(x) d(x)^{\frac{\beta}{\beta -\alpha -1}}=0\, ,
$$
then ergodic pairs $(c_\Omega, u)$ exist, and any ergodic function $u$ satisfies
\begin{equation}\label{asym1}
\lim_{d(x)\to 0} \frac{u(x)\, d(x)^\chi}{C(x)}=1  \mbox{ if}\  \chi>0, 
\end{equation}
and 
\begin{equation}\label{asym2}
 \lim_{d(x)\to 0} \frac{u(x)}{|\log d(x)|\, C(x)}=1   \hbox{ if}\  \chi=0\, ,
\end{equation}
where
$$\chi = \frac{ 2+ \alpha-\beta }{ \beta-1-\alpha }\, ,$$
and, for $x$ in a neighborhood of $\partial \Omega$,    
\begin{equation}\label{C(x)}
\begin{array}{ccc}C(x)=\left((\chi+1)F(\grad d(x)\otimes\grad d(x))\right)^{\frac{1}{ \beta-\alpha-1}}\chi^{-1}\  &\mbox{if }& \chi>0,\\[1ex]
 C(x)=F(\grad d(x)\otimes\grad d(x))\quad & \mbox{if }& \chi=0.
 \end{array}
\end{equation}
Let us now observe that any ergodic function verifies the assumptions of  Theorem \ref{C1reg}, so that it is  a $C^{1,\gamma}_{loc}(\Omega)$ function and satisfies the local a priori bound \eqref{bound1}. This regularity property, jointly with the asymptotic estimates \eqref{asym1},  allows us to precise, at least in the case $\chi>0$,  the   boundary asymptotic behaviour of its gradient. We obtain the analogous of the  result proved in \cite{P} for   Laplace operator and in  \cite{LP} for   $p$-Laplace operator.
\begin{theo}\label{asymgrad} Assume that $F$ satisfies \eqref{eqF1} and \eqref{C2reg},  let $f\in {\mathcal C}(\Omega)$ be  bounded and suppose that $\chi>0$ (i.e. $\beta<\alpha+2$). Then, for   any ergodic function $u$, one has
 \begin{equation}\label{asymgrad1}
 \lim_{d(x)\to 0}   \frac{d(x)^{\chi+1} \nabla u(x)\cdot \nabla d(x)}{C(x)}=
 -\chi \, .
 \end{equation}
\end{theo}
  
\begin{proof}
Let us consider, for  $x_0\in \partial \Omega$ fixed and $\delta>0$, the function
$$
u_\delta (\zeta)=\delta^\chi u(x_0+\delta\, \zeta)\, ,
$$
defined for $\zeta \in \frac{1}{\delta}\left( \Omega -x_0\right)$.

By \eqref{asym1}, one has
$$\lim_{\delta\to 0}  u_\delta (\zeta)= \frac{C(x_0)}{\left( \nabla d(x_0)\cdot \zeta\right)^\chi}
$$
locally uniformly with respect to $\zeta$ in the halfspace $H=\{ \zeta \in \R^N\, :\ \zeta \cdot \nabla d(x_0)>0\}$, and uniformly with respect to $x_0\in \partial \Omega$. In particular, $u_\delta$ is locally uniformly bounded in $H$.

Moreover, by direct computation, $u_\delta$ satisfies the equation
$$
-|\nabla u_\delta|^\alpha F(D^2 u_\delta) + |\nabla u_\delta|^\beta = \delta^{\frac{\beta}{\beta-\alpha-1}} \left[ f(x_0+\delta\, \zeta)+c\right]\quad \hbox{ in } \frac{1}{\delta}\left( \Omega -x_0\right)\, .
$$
Thus, as a consequence of Theorem \ref{C1reg}, $u_\delta$ belongs to $\mathcal{C}^{1,\gamma}_{loc}(H)$ and verifies  estimate \eqref{bound1}. This implies that $u_\delta$ is  converging in $\mathcal{C}^1_{loc}(H)$, and, therefore,
$$
\lim_{\delta \to 0} \nabla u_\delta (\zeta)=- \chi \frac{C(x_0)\, \nabla d(x_0)}{\left(\nabla d(x_0)\cdot \zeta\right)^{\chi+1}}
$$
locally uniformly with respect to $\zeta\in H$ and, again, uniformly with respect to $x_0\in \partial \Omega$.

Hence, we  deduce that
$$
\lim_{\delta \to 0} d(x_0+\delta\, \zeta)^{\chi+1} \nabla u (x_0+\delta\, \zeta)= -\chi\, C(x_0)\, \nabla d(x_0)
$$
locally uniformly with respect to $\zeta\in H$ and uniformly with respect to $x_0\in \partial \Omega$. This immediately yields \eqref{asymgrad1}. 
\end{proof}

\begin{rema}
{\rm In the case $\chi=0$, one can try to use an analogous argument as above, and to  consider the function
$$ u_\delta ( \zeta) = u( x_0+ \delta \zeta)+ C(x_0)\log ( \delta)\, . $$ 
By Theorem 4.2 and Theorem 6.3 of \cite{BDL1}, it follows that $u_\delta$  is uniformly bounded. Moreover, arguing as in the above proof,  we obtain that $u_\delta$ actually converges in $\mathcal{C}^1_{loc}(H)$  to a solution of 
$$\left\{ \begin{array}{cl}
 -| \nabla v |^\alpha  F( D^2 v) + | \nabla v |^{2+ \alpha} = 0 & \hbox{ in } \ H\\
  v = +\infty & \hbox{ on } \partial H
  \end{array}\right.$$
    Using the  same argument as in Section 4 of  \cite{IS}, one gets that $v$ satisfies 
    $$\left\{ \begin{array}{cl}
 -  F( D^2 v) + | \nabla v |^2 = 0 & \hbox{ in } H\\
  v= +\infty & \hbox{ on } \partial H
  \end{array}\right.$$
Now, consider first the case when  $F$ is a \emph{linear operator},  that is $F(M) = a\,  {\rm tr} (M)$. Then,  defining $ \varphi = e^{-v/a}$ , one sees that $\varphi$ is positive and  harmonic   in $H$, and it satisfies   zero  boundary conditions. Hence,  $\varphi (\zeta) = c\, \nabla d(x_0)\cdot \zeta$ for some constant $c>0$. Coming back to $v$,  one gets that 
$v(\zeta)  =  -a\, \log \nabla d(x_0)\cdot \zeta - a\, \log c$, and,
 by the local  ${\cal C}^1$ convergence of $u_\delta$ to $v$, we conclude that 
$$
\lim_{\delta\to 0} \delta\, \nabla u(x_0+\delta\, \zeta)=- \frac{ a\, \nabla d(x_0)}{\nabla d(x_0)\cdot \zeta}\, .
$$
 Observing that in this case $a=C(x_0)$, we deduce  for linear operators the asymptotic gradient behaviour
 $$
 \lim_{d(x)\to 0}   \frac{d(x) \nabla u(x)\cdot \nabla d(x)}{C(x)}=
 -1 \, ,
$$
which is the analogous of \eqref{asymgrad1} for $\chi=0$. 

For general $F$,  an analogous result could be obtained as a consequence of the following Liouville type result:   if  $u$ is a solution in the half space $H=\{ x_N >0\}$ of 
  $$\left\{ \begin{array}{lc}
   -F( D^2 u) + | \nabla u |^2 = 0 & {\rm in} \ \{x_N >0\}\\
    u( x^\prime, 0) = +\infty, &
\end{array} \right.$$
 then there exists some constant $c$ such that 
 $$u (x) = F( e_N \otimes e_N) | \log x_N|+ c \, . $$
 By the time being, this is an open question. }
 \end{rema}

We are finally  in the position to prove the uniqueness, up to additive constants, of the ergodic function.
\medskip

\noindent  \emph{Proof of Theorem \ref{unique}.} By Theorem \ref{asymgrad},  any ergodic function $u$ satisfies the asymptotic gradient boundary behaviour \eqref{asymgrad1}.         
          Hence, there exists a positive  constant $C$ such that 
          $| \nabla u | \geq C d^{ -\chi-1}$ for $d < \delta$ . We can suppose that $\delta$ is so small that    $( \beta-1-\alpha) C d^{-( \chi+1) \beta} > 2 |f|_\infty  (1+ \alpha)  $.    
         
           Suppose now that  $u$ and $v$ are two  ergodic functions related to the same ergodic constant $c_\Omega$.   Recall that 
          $ \Omega_\delta= \{ x \in \Omega, d( x) < \delta\}$ and consider 
           $u_\epsilon = (1-\epsilon) u$.
             Let for further computations $c_\alpha$ and $c_\beta$ some positive constants so that for $\epsilon < {1\over 2}$, 
           $$ |( 1-\epsilon )^{1+ \alpha}-1+(1+ \alpha) \epsilon | \leq c_\alpha \epsilon^2$$
            $$ |( 1-\epsilon )^{\beta}-1+\beta \epsilon | \leq c_\beta \epsilon^2$$
             and take 
             $\epsilon < { \beta-\alpha -1 \over 4} \inf_{ \Omega_\delta}{ | \nabla u |^\beta \over ( c_\alpha+ c_\beta) ( | \nabla u |^\beta + |f+ c |_\infty)}$. 
         Then,  $u_\epsilon$ is  a strict sub-solution in $\Omega_\delta$.  Indeed 
          \begin{eqnarray*}
            -| \nabla u_\epsilon |^\alpha F( D^2 u_\epsilon) &+& | \nabla u_\epsilon |^\beta -(f+ c_\Omega) \\
            &=& (1-\epsilon)^{1+ \alpha} \left( -| \nabla u |^\alpha F( D^2 u) + | \nabla u |^\beta -(f+ c_{\Omega})\right)\\
            &+& ((1-\epsilon)^\beta -(1-\epsilon )^{1+ \alpha}) | \nabla u |^\beta + (f+ c_\Omega) ((1-\epsilon )^{1+ \alpha}-1)\\
            &\leq & - \epsilon ( \beta-1-\alpha) C d^{-( \chi+1) \beta} + |f+ c_\Omega |_\infty \epsilon (1+ \alpha)\\
            &+& \epsilon ^2 (c_\alpha + c_\beta) \left( | \nabla u |^\beta + |f+c|_\infty\right)\\
            &<&0
            \end{eqnarray*}
            By the asymptotic behavior both of $u$ and $v$, let  $V_\epsilon$ a neighborhood of the boundary on which $u_\epsilon < v$. Applying 
           the comparison principle in $\Omega_\delta \setminus \overline{V_\epsilon}$ (see \cite{BD1}), when one of the sub- (super-) solution is strict,  one gets that  
             $u_\epsilon \leq v+ \sup_{ d = \delta } ( u_\epsilon -v)$  in $\Omega_\delta \setminus \overline{V_\epsilon}$, hence finally in the whole of $\Omega_\delta$, 
              Passing  to the limit one gets that 
              $ u \leq v + \sup_{ d= \delta} ( u-v)$  in $\Omega_\delta$. On the other hand,  using the comparison principle without zero order terms when $\sup ( f + c_\Omega )<0$,   proved in \cite{BDL1}, one gets that 
              $ u \leq v+\sup_{ \partial \Omega_\delta} ( u-v)$  in $\Omega\setminus \overline{ \Omega}_\delta $. We need to prove that $u $ indeed  coincides with  $v+ \sup_{ d= \delta} ( u-v):= v+m$. The step before says that the supremum of $(u-v)$ in $\Omega$ is achieved on $d = \delta  $,  hence inside $\Omega$. 
               When $\alpha = 0$, the strong maximum principle implies that $ u = v+ m$. 
              
              When $\alpha \neq 0$,  suppose that $\partial \Omega$ has only one connected component, then $\Omega_\delta$ is connected. 
                          We want to  prove that  in the whole of $\Omega_\delta$, $u = v+    m$.  Indeed note that $\Omega_\delta$ has been chosen so that $| \nabla u | \geq C d^{-\chi-1}$  inside it,  hence in particular $\nabla u\neq 0$ in $\Omega_\delta$.    Then there is an   $\bar x\in \Omega$, such that                $ u( \bar x) = v( \bar x)+ m$, $u\leq v+ m$, and $\nabla u( \bar x)  = \nabla  v( \bar x) \neq 0$.                 Using the strong comparison principle in \cite{BD3} one gets that there exists a neighborhood $V_{ \bar x}$ of $\bar x$ where 
                $ u\equiv v+ m$. 
                 Denote 
                 ${ \cal O}_\delta = \{ x \in \Omega_\delta, u( x) = v(x)+ m\}$. By the previous argument there is one ball $B( \bar x, r)$ so that 
                 $ B( \bar x, r) \cap \Omega_\delta \subset { \cal O}_\delta$. In particular ${ \cal O}_\delta$ is non empty, and  the same argument proves that ${ \cal O}_\delta$ is open.  By definition it is closed, so ${ \cal O}_\delta= \Omega_\delta$. Then applying the comparison Theorem without zero order terms in $\Omega \setminus \overline{ \Omega}_\delta$, since we have on its boundary 
                  $ u = v+m$ one gets both that $u \leq v+ m$ and $u \geq v+ m$. Finally $u = v+ m$.   
                      
\hfill$\Box$

\noindent                    {\bf Acknowledgment. }
                    This work was done  while the first and third author  were visiting the University of Cergy -Pontoise,  and  the second one the University of Roma 1, supported by INdAM-GNAMPA.

\end{document}